\documentclass[runningheads,a4paper,envcountsame]{llncs}

\usepackage{amsfonts}
\usepackage{amssymb}
\usepackage{mathtools}
\usepackage[all,cmtip]{xy}
\usepackage[bookmarks=false]{hyperref}
\usepackage{xcolor}
\usepackage{todonotes}
\usepackage{url}
\newtheorem{defini}{Definition}
\newtheorem{notation}{Notation}

\newtheorem{case_dist}{Case}

\usepackage{breqn}

\renewcommand{\phi}{\varphi}

\usepackage{todonotes}
\usepackage{tikz}

\title{Reduction Complexities in Set Theory}
\author{Merlin Carl}
\date{July 2025}
\institute{Institut f\"ur Mathematik, Europa-Universit\"at Flensburg}

\begin{document}

%\todo{add references to Stammes' thesis. new topic: parallel vs. sequential requests (known from Weiharuch complexity): you can make as many \textit{simulataneous} requests as you want, we only care about nesting -- what then? probably not much of a change for most of the result, but worthy of investigation...}

\maketitle

\begin{abstract}
    In \cite{Ca2016} and \cite{Ca2018}, we introduced a notion of effective reducibility between set-theoretical $\Pi_{2}$-statements; in \cite{Ca2025}, this was extended to statements of arbitrary (potentially even infinite) quantifier complexity. We also considered a corresponding notion of Weihrauch reducibility, which allows only one call to the effectivizer of $\psi$ in a reduction of $\phi$ to $\psi$. In Stammes \cite{StammesMaster}, a considerably refined analysis through interpolating between these two notions was proposed, where one asks how many calls to an effectivizer for $\psi$ are required for effectivizing $\phi$. This allows us to make formally precise questions such as ``how many ordinals does one need to check for being cardinals in order to compute the cardinality of a given ordinal?'' and (partially) answer many of them. Many of these anwers turn out to be independent of ZFC. %We also introduce a parallelization operator analogous to the one used in Weihrauch reducibility and distinguish the number of nested calls from the number of parallel calls.
\end{abstract}

\section{Introduction and Basic Notions}

Imagine you have a transfinite life span, are given an ordinal $\alpha$ and you want to know its cardinality. You have a black box, which, given an ordinal $\beta$, will answer ``yes'' or ``no'', depending on whether or not $\beta$ is a cardinal. Then you can solve your problem by checking all ordinals below $\alpha$ then taking the supremum of all those that turn out to be cardinals. This method requires $\alpha$ many (in terms of the order type) accesses to the black box. Can you do any better?

In classical computability theory, such questions are treated under the name of ``bounded queries'', see, e.g., Martin and Gasarch, \cite{Gasarch1}, Gasarch and Stefan \cite{Gasarch2}, and Gasarch \cite{Gasarch3}.\footnote{We thank Vasco Brattka for pointing out this reference to us.} In transfinite computability, such a notion was proposed in Stammes \cite{StammesMaster}, Definition 4.2.

To answer questions such as the one just formulated, this needs to be generalized to the transfinite. In line with \cite{Ca2016}, \cite{Ca2018} and \cite{Ca2025}, we will use Koepke's ordinal Turing machines (\cite{Koepke}) to model the idea of a transfinite ``method''. 

More generally, in this paper, we want to measure the relative complexity  -- which we call ``reduction complexity'' -- of certain functions embodying natural set-theoretical principles (such as ``every set is equivalent to a cardinal'') by the number of calls to one function that one needs in order to compute (on an ordinal Turing machine) another. 

\subsection{Basic Definitions and Notation} 

The model of computation underlying this paper are Koepke's ordinal Turing machines (OTMs), introduced in \cite{Koepke}. A pair $(P,\rho)$ of an OTM-program and an ordinal parameter $\rho$ is called a \textit{parameter-program}. If we imagine the OTM-programs to be enumerated in order type $\omega$ in some natural way, we can identify a program $P_{k}$ with its index $k$ and encode a parameter-program $(P_{k},\rho)$ by a single ordinal $\gamma:=\omega\rho+k$; below, we will occasionally use this to simplify our notation when program and parameter do not need to be considered separately.

Although the function types considered in this paper can be regarded as types of ``effectivizers'' (in the sense of \cite{Ca2025}) of certain set-theoretical statements that were considered in \cite{Ca2025}, it saves some technical details to define them directly here.

\begin{itemize}
    \item A Pot function is a class function $F:\mathfrak{P}(\text{On})\rightarrow\mathfrak{P}(\text{On})$ that maps every encoding of a set to an encoding of its power set. 
    \item A PowerCard function is a class function $F:\mathfrak{P}(\text{On})\rightarrow\text{On}$ that maps every encoding of a set to the cardinality of its power set.
    \item A NextCard function is a class function $F:\text{On}\rightarrow\text{On}$ that maps every ordinal $\alpha$ to its cardinal successor $\alpha^{+}$. 
    \item An OrdCard function is a class function $F:\text{On}\rightarrow\text{On}$ that maps every ordinal $\alpha$ to its cardinality $\text{card}(\alpha)$. 
    \item DecCard denotes the class function $F:\text{On}\rightarrow\{0,1\}$ that is defined by $$F(\alpha)=\begin{cases}1\text{, if }\alpha\text{ is a cardinal},\\0\text{, otherwise}\end{cases}.$$
    \item DecReg denotes the (unique) class function $F:\text{On}\rightarrow\{0,1\}$ that is defined by $$F(\alpha)=\begin{cases}1\text{, if }\alpha\text{ is a regular cardinal},\\0\text{, otherwise}\end{cases}.$$
    \item CardExp denotes cardinal exponentiation, i.e. the (unique) class function that maps an ordered pair $(\kappa,\lambda)$ of cardinals to the cardinal $\kappa^{\lambda}$. 
    \item Cof denotes the cofinality function, i.e., the (unique) class function that sends each ordinal to its cofinality. 
    \item  For $n\in\omega$, a $\Sigma_{n}$-Sep function is a class function $F:\mathfrak{P}(\text{On})\times\omega\times\mathfrak{P}(\text{On})\rightarrow\mathfrak{P}(\text{On})$ that maps a triple $(c(S),k,c(\vec{p}))$ consisting of an encoding of a set $S$, an index $k$ for a $\Sigma_{n}$ $\in$-formula $\phi_{k}$ and an encoding of a finite tuple $\vec{p}$ of sets to an encoding of the set $\{x\in S:\phi_{k}(x,\vec{p})$ if $\vec{p}$ has the right length, and to $\emptyset$, otherwise.
    \item For $n\in\omega$, a $\Sigma_{n}$-Truth function is a class function $F:\omega\times\mathfrak{P}(\text{On})\rightarrow\{0,1\}$ that maps a pair $(k,c(\vec{p}))$ consisting an index $k$ for a $\Sigma_{n}$ $\in$-formula $\phi_{k}$ and an encoding of a finite tuple $\vec{p}$ of sets to $0$ or $1$, according to the following condition: $$F(k,c(\vec{p}))=\begin{cases}1\text{, if }\phi_{k}(\vec{p})\\0\text{, otherwise}\end{cases}.$$
    
\end{itemize}

Note that, due to the possibility of different encodings, these are function \text{types} rather than particular functions, although for the types OrdCard, NextCard, PowerCard and DecCard, there is only one (class) function that belongs to them. Since the functions we consider are proper classes, these types cannot be introduced as objects in ZFC. One way to formalize the work below in ZFC is via talking about properties of formulas instead. We will not go into the details of the formalization.

In agreement with the definitions in \cite{Ca2025}, we say that one function type $A$ is OTM-reducible to another function type $B$, written $A\leq_{\text{OTM}}B$ if and only if there is a parameter-OTM-program $(P,\rho)$ such that, for each function $F$ of type $B$, $P^{F}(\rho)$ computes a function of type $A$. If this computation makes at most one call to $F$ for each input, we say that $A$ is ordinal Weihrauch reducible to $B$ and write $A\leq_{\text{oW}}B$.

The gap between OTM-reducibility and oW-reducibility is rather wide: In the case of the former, we allow any number of calls to the external function (i.e., the ``effectivizer''), while in the latter, only a single one is allowed. 
In \cite{StammesMaster}, Stammes proposed a more refined notion, differentiating reductions by the required number of calls to the extra function.  To this end, we fix the following definition:\footnote{Stammes only considered constant bounds on the number of calls to the extra function; the definition below is thus a slight definition of his.} 
%In this paper, we work towards 
%a more refined notion, differentiating reductions by the required number of calls to the extra function.  To this end, we fix the following definition: 

\begin{defini} (Cf. Stammes, \cite{StammesMaster}, Definition 4.2)

Let $\Phi$ and $\Psi$ be types of (class) functions, and let $f:V\rightarrow\text{On}$ be a (class) function. We say that $\Phi$ is $f$-OTM-reducible to $\Psi$ if and only if is a parameter-program $(P,\rho)$ which OTM-reduces $\Phi$ to $\Psi$ and, for any instance $a$ and any $F$ of type $\Psi$, the order type of calls to $F$ in the computation $P^{F}(a,\rho)$ is at most $f(a)$. We denote this by $\Phi\leq_{\text{OTM}}^{f}\Psi$. In particular, if $f$ is constant with value $\xi$, we write $\Phi\leq_{\text{OTM}}^{\xi}\Psi$. 

If the order type of the set of times at which calls to the extra function take place is strictly below $f(x)$ for all but set many $x$, we write $\Phi\leq_{\text{OTM}}^{<f}\Psi$ and say that $f$ is an upper bound to the reduction complexity of $\Phi$ to $\Psi$. If, on the other hand, the number of calls to the extra function is at least $f(x)$ for all but set many inputs $x$, we write $\Phi\leq_{\text{OTM}}{\geq f}\Psi$ and say that $f$ is a \textit{lower bound} for the reduction complexity of $\Phi$ to $\Psi$.

\end{defini}

At this point, one might wonder why we are using the order type of the calls to the extra function $F$ in the computation, rather than the cardinality of the set of times at which such calls take place, as a measure for ``how often'' $F$ is accessed. Intuitively, it is clear that the order type needs a finer measure, as, e.g., in a sequence of $\omega+1$ calls, the instance to which the last call is applied is informed by infinitely many instances of $F$, which is not possible in a sequence of calls of order type $\omega$. 
The following example -- which is admittedly a bit artificial -- illustrates this point (thereby answering Question $2$ in \cite{StammesMaster}): 

\begin{proposition}
    For $\xi\in\text{On}$, let a $\xi$-NextCard function be a function that maps an ordinal $\alpha$ to the $\xi$-th cardinal greater than $\alpha$ (thus, for example, an $\omega$-NextCard function would map $\aleph_{0}$ to $\aleph_{\omega}$). Then, for any $\xi$, we have $\xi-\text{NextCard}\leq_{\text{OTM}}^{\xi}\text{NextCard}$, while $\xi-\text{NextCard}\nleq_{\text{OTM}}^{\zeta}\text{NextCard}$ for all $\zeta<\xi$.
\end{proposition}
\begin{proof}
    The reduction $\xi-\text{NextCard}\leq_{\text{OTM}}^{\xi}\text{NextCard}$ is obvious by iteratively applying NextCard to a given ordinal $\alpha$ for $\xi$ many times, taking unions at limits. (To be able to perform $\xi$ applications of NextCard, we can put $\xi$ in the parameter of our program.)

    To see that the reduction cannot work with fewer than $\xi$ applications, we show by induction on $\iota$ that a halting OTM-program starting on a tape on which at most the first $\alpha$ many cells contain $1$s will, with $\iota$ many applications of NextCard, halt in a state where the set of cells on which it has written a $1$ is bounded above by some $\lambda<\aleph_{\beta+\iota+1}$, where $\text{card}(\alpha)=\aleph_{\beta}$. 

    For $\iota=0$, this is folklore\footnote{For OTMs, see, e.g. \cite{CarlBuch}, proof of Lemma 8.6.3. The analogous statement for Infinite Time Turing Machines was already noted and proved in  Hamkins and Lewis, \cite{Hamkins-Lewis}, Theorem 1.1.} 
    
    For $\iota=\delta+1$, we will, after the $\delta$-th call to $F$, have written on a set of cells bounded above by some $\lambda < \aleph_{\beta+\delta}$. Now, as in the case $\iota=0$, before the next call to NextCard takes place, the machine will make less than $\aleph_{\beta+\delta+1}$ many steps (to apply the case $\iota=0$ here, imagine the machine changed to halt when the original program would have called $F$), resulting on a tape on which the indices of the cells written on is again bounded above by some $\delta^{\prime}<\aleph_{\beta+\delta+1}$. Applying NextCard will then lead to a tape on which these indices are bounded by some $\delta^{\prime\prime}<\aleph_{\beta+\delta+2}$, and without any further calls to $F$, it must halt in $\delta^{\prime\prime\prime}<\aleph_{\beta+\delta+2}=\aleph_{\beta+\iota+1}$ many steps, and thus the set of indices of cells it wrote on is bounded above by $\aleph_{\beta+\iota+1}$. 
    
    Finally, when $\iota$ is a limit ordinal, then the supremum of the cells written on right before the $\iota$-th call is the supremum of the suprema of these sets for all $\zeta<\iota$, which, by assumption, are all strictly smaller than $\aleph_{\beta+\iota}$. Thus, with the $\iota$-th call, the set of these indices is bounded above by $\aleph_{\beta+\iota+1}$, and it remains so when another halting OTM-computation is performed on it. 

    Thus, with fewer than $\xi$ many applications of NextCard, we cannot even reach enough cells to write the result of $\xi$-NextCard.
\end{proof}

As a consequence, we have reductions that work with $\omega+1$ many calls to $F$, but not with $\omega$ many calls, with $\omega+2$ many calls, but not with $\omega+1$ many calls etc. Thus, the order type is the more sensible measure of complexity here, rather than mere cardinality. 

\section{Basic Tools}

In this section, we gather some basic observations about reduction complexity, along with some lemmata that will help in proving concrete results later on.

\begin{proposition}{\label{transitivity}} (Cf. Stammes, \cite{StammesMaster}, Proposition 4.5) 
If $\Phi\leq_{\text{OTM}}^{f}\Psi\leq_{\text{OTM}}^{g}\Gamma$, then $\Phi\leq_{\text{OTM}}^{f\cdot g}\Gamma$.
\end{proposition}

While upper bounds for reduction complexities can be read off from concrete constructions, lower bounds require more work. Currently, our best tool for lower bounds is Corollary \ref{cardinal time bound} below.

\begin{notation}
For any ($F$-)OTM-program $P$, any set  $a$ and any sequence $\vec{s}:=(s_{\iota}:\iota<\xi)$ of sets of ordinals, denote by $P^{F\rightarrow\vec{s}}(a)$ the computation that is obtained when, for any $\iota<\xi$, the $\iota$-th call that $P$ makes to $F$ is answered with $\vec{s}(\iota)$. 
\end{notation}

Thus, if $\xi$ is the order type of all calls that $P^{F}(a)$ makes to $F$ and $\vec{s}$ is the sequence of values that $F$ returns at these calls, then $P^{F\rightarrow\vec{s}}(a)$ is the same computation (as a sequence of machine states) as $P^{F}(a)$. 
In particular, the computation of $P^{F}(a)$ can be OTM-effectively obtained from $\vec{s}$ and $a$.

\begin{lemma}{\label{no repeated requests}}
For each parameter-program $\gamma$, there is a parameter-program $\tilde{\gamma}$ such that, for each class function $F:V\rightarrow V$, $\tilde{\gamma}^{F}$ computes the same function as $\gamma^{F}$, but with the extra property that $\tilde{\gamma}$ never calls $F$ twice on the same argument.
\end{lemma}
\begin{proof}
    $\tilde{\gamma}$ runs like $\gamma$, but uses an extra tape $S$. Whenever $\gamma$ makes a call to $F$ to determine $F(x)$, $\tilde{\gamma}$ first checks whether some pair of the form $(x,y)$ is contained on $S$. If it is, then $y$ is retrieved from $S$ and the call to $F$ is not executed. Otherwise, the call to $F$ is executed, and the pair $(x,F(x))$ is stored on $S$. 
\end{proof}

\begin{defini}
    For $k\in\omega$ and $a$ a set, let $\sigma_{k}(a)$ be the minimal ordinal $\alpha$ such that $L_{\alpha}[a]\prec_{\Sigma_{k}}L[a]$. 
\end{defini}

\begin{remark}
If $a$ is transitive, let $H$ be the $\Sigma_{n}$-hull of $a\cup\{a\}$ in $L[a]$; forming the transitive collapse $\overline{H}$ of the result will leave $a$ fixed, so that, by the condensation lemma, we have $\overline{H}=L_{\xi}[a]$ for some ordinal $\xi$; by definition, $L_{\xi}[a]\prec_{\Sigma_{n}}L[a]$. Thus, whenever $a$ is transitive, $\sigma_{k}(a)$ exists (and these are the only cases we care about in this paper). Moreover, the cardinality of $H$ in $L[a]$ is bounded above by $\text{card}^{L[a]}(a)\times\aleph_{0}$, so, when $a$ is infinite, we will have $\text{card}^{L[a]}(\xi)=\text{card}^{L[a]}(a)$ and in particular $\xi< (\text{card}(a))^{+}$. 
\end{remark}

\begin{lemma}{\label{cardinality and reduction complexity}}
 Let $P$ be an OTM-program, $\rho\in\text{On}$ a parameter, 
 let $F$ be a 
 class function and $t$ an initial tape content (i.e., an input, which we assume to be set-sized, i.e., bounded).  Let $(\tau_{\iota}:\iota<\xi)$ be the sequence of times at which $F$ is called, and let $\vec{v}:=(v_{\iota}:\iota<\xi)$ be the sequence of values returned by $F$ at these times.
 
  Then $\tau_{\gamma}<\sigma_{1}(\{t\}\cup(\rho+1)\cup \text{tc}(\{\vec{v}\upharpoonright\gamma\}))$, and thus, in particular, $\tau_{\gamma}<\sigma_{1}(\{t\}\cup(\rho+1)\cup \text{tc}(\{\vec{v}\}))$ for all $\gamma<\xi$.

  Finally, if $P^{F}(t,\rho)$ halts, then its halting time is also strictly smaller than $\sigma_{1}(\{t\}\cup(\rho+1)\cup \text{tc}(\{\vec{v}\}))$. 

\end{lemma}
\begin{proof}
    Let $\gamma<\xi$, and let $z$ be the computation state (tape content and inner state) of $P^{F}(t,\rho)$ at time $\tau_{\gamma}+1$, i.e., right after the $\gamma$-th call to $F$. Then $z$ is $\Sigma_{1}$-definable in $t$, $\rho$ and $\gamma$, and, since $\gamma$ is $\Sigma_{1}$-definable from $\vec{v}\upharpoonright\gamma$, it is an element of $L_{\sigma_{1}(\{t\}\cup(\rho+1)\cup \text{tc}(\{\vec{v}\}))}[t]$. 

    Now, the statement that $P$, when run on the initial configuration $z$, makes at least one call to the extra function is $\Sigma_{1}$, and (by assumption) true in $V$ and thus in $L[\text{tc}(\{\vec{v}\upharpoonright\gamma\})]$. Consequently, it must be true in $L_{\sigma_{1}(\{t\}\cup(\rho+1)\cup \text{tc}(\{\vec{v}\upharpoonright\gamma\}))}$. But then, the $\gamma$-th call must take place before time $\sigma_{1}(\{t\}\cup(\rho+1)\cup \text{tc}(\{\vec{v}\upharpoonright\gamma\}))$, which is what we wanted. 

   The second claim now follow immediately. For the third claim, just note that, if all calls to $F$ in the computation of $P^{F}(t,\rho)$ are contained in $\vec{v}$, then the statement ``There is a time at which $P^{F}(t,\rho)$ halts'' is $\Sigma_{1}$ in $\vec{v}$, $t$ and $\rho$, and thus the same argument as above shows that the halting time must be below $\sigma_{1}(\{t\}\cup(\rho+1)\cup \text{tc}(\{\vec{v}\}))$.
 
\end{proof}

\begin{corollary}{\label{cardinal time bound}}
Suppose that $F:V\rightarrow V$ is a  
class function such that $\text{card}(\text{tc}(F(x)))\leq\text{card}(\text{tc}(x))$ for  all sets $x$ (i.e., in the terminology of Hodges \cite{Hodges}, $F$ does not ``raise cardinalities''), let $P$ be an OTM-program, $\rho\in\text{On}$, $t\subseteq\text{On}$ a set of ordinals (encoding the initial tape content) and $\kappa>\rho,\text{sup}(t)$ an uncountable cardinal. 

Then, if $P^{F}(s,\rho)$ halts in more than $\kappa$ many steps, it makes $\kappa$ many calls to $F$ before time $\kappa$. 

Moreover, if $P^{F}(s,\rho)$ computes for at least $\kappa$ many steps and makes less than $\kappa$ many calls to $F$ before time $\kappa$, then it will diverge without making any further calls to $F$ at or after time $\kappa$. 
\end{corollary}
\begin{proof}
The computation clearly cannot make more than $\kappa$ many calls to $F$ before time $\kappa$; we thus only need to show that it cannot make less than that many calls.

First, let us assume that $\kappa$ is regular. Suppose for a contradiction that $P^{F}(t,\rho)$ makes $\gamma<\kappa$ many calls to $F$ before time $\kappa$. As in Lemma \ref{cardinality and reduction complexity}, let $(\tau_{\iota}:\iota<\gamma)$ be the times before $\kappa$ at which calls to $F$ were made during this computation and let $\vec{v}:=(v_{\iota}:\iota<\gamma)$ be the values returned by $F$ at these requests. By regularity of $\kappa$, we have $\delta:=\text{sup}\{\tau_{\iota}:\iota<\gamma\}<\kappa$. By induction, we have $\text{card}(\text{tc}(v_{\iota}))<\kappa$ for all $\iota<\gamma$: At successor levels, this follows from the assumption on $F$, while at limit levels $\delta$, this follows from the regularity of $\kappa$: If tape portions written on before time $\delta$ are always of length strictly smaller than $\kappa$, then the tape portion written on at time $\delta$, being bounded above by the union of the lengths of these tape portions, will, as a union of strictly less than $\kappa$ many ordinals strictly smaller than $\kappa$, also be strictly smaller than $\kappa$.

It follows that $\text{card}(\text{tc}(\vec{v}))<\kappa$, so $\vec{v}\in H_{\kappa}^{L[\vec{v}]}=L_{\kappa}[\vec{v}]$. Now, if no further calls to $F$ are made at all after time $\delta$ (including times $\geq\kappa$), then it follows from the remark above that $\sigma_{1}(\text{tc}(\{t\})\cup\vec{v}\cup(\rho+1))<\kappa$, so, by Lemma \ref{cardinality and reduction complexity},  the computation will halt in less than $\kappa$ many steps, a contradiction. Thus, in this case, there must be at least one call to $F$ taking place at time $\tau\geq\kappa$.

Let $\phi(t,\delta,\vec{v},\rho)$ be the statement ``There is a time strictly above $\delta$ at which $P^{F\rightarrow\vec{v}}(t,\rho)$ makes a call to $F$''. 
Then $\phi$ is $\Sigma_{1}$ (note that the value returned by $F$ at this call is irrelevant to the truth of this statement). Since this statement is true in $L[\text{tc}(\{\vec{v}\})]$, it must be true in $L_{\sigma_{1}(\text{tc}(\{t\})\cup\{\vec{v}\}\cup(\rho+1))}[\text{tc}(\vec{v})]$ (which contains all occuring parameters). But then, there must be a call to $F$ between times $\delta$ and $\sigma_{1}(\text{tc}(\{t\})\cup\text{tc}(\{\vec{v}\})\cup(\rho+1))<\kappa$, contradicting the definition of $\delta$.

Now, if $\kappa$ is singular, we can write it as a union of an increasing sequence $(\kappa_{\iota}:\iota<\gamma)$ of regular cardinals. Since $\kappa>\rho,\text{sup}(v)$, there is $\xi<\gamma$ such that $\kappa_{\iota}>\rho,\text{sup}(v)$ for $\iota\geq \xi$; without loss of generality, assume that $\xi=0$. If $P^{F}(t,\rho)$ halts in more than $\kappa$ many steps, then, in particular, for every $\iota<\gamma$, it halts in $>\kappa_{\iota}$ many steps and thus, before time $\kappa_{\iota}$, makes at least $\kappa_{\iota}$ many calls to $F$. Since this is true for all $\iota<\gamma$, it will make $\kappa$ many calls to $F$ before time $\kappa$. 

We now show the second claim. Suppose first that $P^{F}(s,\rho)$ makes actually less than $\text{cf}(\kappa)\leq\kappa$ many calls to $F$. From what we just showed, it follows that $P^{F}(s,\rho)$ will not halt. To see that there will be no calls to $F$ at or after time $\kappa$, let $\delta$ be the supremum of times at which such calls are made before time $\kappa$; by assumption, we have $\delta<\kappa$. Let $z$ be the computation state of $P^{F}(s,\rho)$ at time $\delta$. Consider a slightly modified version $Q$ of the program $P$ which terminates, when $P$ makes a call to $F$. Thus, $Q$ is an ordinary OTM-program that makes no calls to an extra function. Consequently, if $Q$ is started on the initial configuration $z$, it will either halt in less than $\sigma_{1}(z)<\kappa$ many steps or not at all. However, if $P$ made calls to $F$ after time $\delta$, then $Q$ would, by assumption, terminate at or after time $\kappa$, a contradiction. 

This implies the second claim immediately if $\kappa$ is regular. If $\kappa$ is singular, let $\xi$ be the order type of the calls made to $F$ in the computation of $P^{F}(s,\rho)$ before time $\kappa$, and pick a regular cardinal $\lambda\in(\text{cf}(\kappa),\kappa)$. Before time $\lambda$, the computation has made $\leq\xi<\lambda=\text{cf}(\lambda)$ many calls to $F$, so the above implies it will in fact not make any further calls to $F$ at or after time $\lambda$, and in particular not at or after time $\kappa$. 
\end{proof}

\section{Several Reduction Complexities}

We will now apply the above tools to various concrete cases.

\subsection{PowerCard and Pot}

We start by considering the question how many applications of Pot are necessary for computing PowerCard (that such a reduction is possible was observed in \cite{Ca2025}).

It is easy to see that PowerCard becomes effective when two uses of Pot are allowed: One for computing the power set, and another one for computing the power set of the power set, which can then be searched for the (code for a) well-ordering of minimal length (note that a well-ordering of any given set $x$ is implicit in its encoding as a set of ordinals). If only one application is allowed, the answer is less obvious. Indeed, the proof that one application of Pot is in general insufficient is considerably more technical.

\begin{lemma}{\label{powercard and pot}}
PowerCard$\leq_{\text{oW}}$Pot is independent of ZFC.
\end{lemma}
\begin{proof}
If $V=L$, then we have PowerCard$\leq_{OTM}$Pot: Given $\alpha\in\text{On}$, use the Pot-function to obtain $\mathfrak{P}(\alpha)$. Now enumerate $L$ until the first $L$-level $L_{\gamma}\ni\mathfrak{P}(\alpha)$ is found. The first such level will be $L_{\text{card}^{L}(\mathfrak{P}(\alpha))+1}$ (since $L_{\text{card}^{L}(\mathfrak{P}(\alpha))}$ is the first $L$-level that contains all constructible subsets of $\alpha$ and over this level, $\mathfrak{P}(\alpha)$ is definable), so $\gamma$ is guaranteed to be a successor ordinal $\gamma=\beta+1$ and we can return $\beta$.

To see that this does not follow from ZFC, let $M$ be a transitive model of ZFC which satisfies $2^{\aleph_{\omega\alpha+1}}=\aleph_{\omega\alpha+4}$ for all ordinals $\alpha$. Such a model can be obtained by Easton forcing (see \cite{Kunen}, S. 265). 
We will obtain a model of ZFC in which  PowerCard$\nleq_{\text{OTM}}$Pot by an iterated (class) forcing which successively sabotages all parameter-programs $(P,\rho)$ that might be candidates for witnessing the reduction.  To this end, let such a pair $(P_{k},\rho)$ be encoded as $\omega\rho+k$. This induces an ordering on these pairs; we will take care of these pairs in this induced ordering. 

We now explain how to obtain, starting in a transitive model $N$ of ZFC in which $2^{\aleph_{\omega\alpha+1}}=\aleph_{\omega\alpha+4}$ for all $\alpha\geq\omega(\omega\rho+k)$, a generic extension $N[G]$ of $N$ in which (i) $2^{\aleph_{\omega\alpha+1}}=\aleph_{\omega\alpha+4}$ for all $\beta>\omega(\omega\rho+k)$ and (ii) $(P_{k},\rho)$ does not witness  the ordinal Weihrauch reduction between PowerCard and Pot. Let $\alpha:=\omega\rho+k$. Let $F$ be a Pot-function in $N$. If $P_{k}^{F}(\aleph_{\omega\alpha+1},\rho)$ does not halt with output $\aleph_{\omega\alpha+4}^{N}$, we take the trivial generic extension $N[G]=N$. 

Otherwise, we need to modify $N$ to ensure that $(P_{k},\rho)$ no longer works. To this end, define $\mathbb{P}_{\rho,k}$ to be the standard forcing for collapsing $\aleph_{\omega\alpha+4}$ to $\aleph_{\omega\alpha+3}$, i.e., the set of  partial functions from $\aleph_{\omega\alpha+3}$ to $\aleph_{\omega\alpha+4}$ of size $<\aleph_{\omega\alpha+2}$. As a successor ordinal, $\aleph_{\omega\alpha+2}$ is regular, so, by \cite{Kunen}, Lemma 6.13, $\mathbb{P}_{\rho,k}$ is $\aleph_{\omega\alpha+2}$-closed for all $\rho$, $k$. Let $N[G]$ be a $\mathbb{P}_{\rho,k}$-generic extension of $N$. 
By \cite{Kunen}, Theorem 6.14, $N[G]$ contains no subsets of $\aleph_{\omega\alpha+1}$ that are not in $N$, so that $\mathfrak{P}^{N}(\aleph_{\omega\alpha+1})=\mathfrak{P}^{N[G]}(\aleph_{\omega\alpha+1})$. Moreover, the forcing will collapse $\aleph_{\omega\alpha+4}^{N}$ to $\aleph_{\omega\alpha+3}^{N}$.

Now, all elements of $\mathbb{P}_{\rho,k}$ have size $\leq\aleph_{\omega\alpha+1}$, so $\aleph_{\omega\alpha+4}\leq |\mathbb{P}_{\rho,k}|\leq \aleph_{\omega\alpha+4}^{\aleph_{\omega\alpha+1}}$. Using the Hausdorff formula (\cite{Jech}, p. 57), the fact that $\kappa^{\lambda}=2^{\lambda}$ for $\kappa\leq\lambda$ for infinite cardinals $\kappa$ and $\lambda$ (\cite{Jech}, Lemma 5.20) and the assumption that $2^{\aleph_{\omega\alpha+1}}=\aleph_{\omega\alpha+4}$, we compute $ \aleph_{\omega\alpha+4}^{\aleph_{\omega\alpha+1}}=\aleph_{\omega\alpha+3}^{\aleph_{\omega\alpha+1}}\cdot\aleph_{\omega\alpha+4}=\aleph_{\omega\alpha+2}^{\aleph_{\omega\alpha+1}}\cdot\aleph_{\omega\alpha+3}\aleph_{\omega\alpha+4}=\aleph_{\omega\alpha+1}^{\aleph_{\omega\alpha+1}}\aleph_{\omega\alpha+2}\aleph_{\omega\alpha+3}\aleph_{\omega\alpha+4} = 2^{\aleph_{\omega\alpha+1}}\aleph_{\omega\alpha+4}=\aleph_{\omega\alpha+4}\cdot\aleph_{\omega\alpha+4}=\aleph_{\omega\alpha+4}$. Since an antichain cannot have more elements than the whole partial order, $\mathbb{P}_{\rho,k}$ satisfies the $\aleph_{\omega\alpha+5}$-cc and thus does not change the cardinals above $\aleph_{\omega\alpha+4}$ (\cite{Kunen}, Lemma 6.9). Moreover, the continuum function above $\aleph_{\omega\alpha+4}$ is also not changed, for, if $\kappa\geq\aleph_{\omega\alpha+4}$, then the number of nice names for subsets of $\kappa$ is bounded by the number of maximal antichains in $\mathbb{P}_{\rho,k}$ to the power of $\kappa$,\footnote{Cf. \cite{Kunen}, p. 209f.}, which, by the above, is bounded above by $(\text{card}(\mathbb{P}_{\rho,k})^{\aleph_{\omega\alpha+4}})^{\kappa}\leq(\aleph_{\omega\alpha+4}^{\aleph_{\omega\alpha+4}})^{\kappa}=\aleph_{\omega\alpha+4}^{\kappa}\leq(2^{\omega\alpha+4})^{\kappa}=2^{\kappa}$.

We now show that, in $N[G]$, $(P_{k},\rho)$ does not oW-reduce PowerCard to Pot. So let $F$ be a Pot-function in $N[G]$. We consider $P_{k}^{F}(\aleph_{\omega\alpha+1},\rho)$. Prior to the application of $F$ in this computation, the cardinality of the number of computation steps is bounded by $\aleph_{\omega\alpha+1}$ (note that, as $\alpha$ is chosen so that $\aleph_{\omega\alpha+1}$ is guaranteed to be strictly larger than $\rho$, the computation will either halt in $<\aleph_{\omega\alpha+1}$ many steps or not at all). Moreover, the cardinality of the transitive closure of the input is also $\aleph_{\omega\alpha+1}$. Thus, the set $S$ to which $F$ is applied in the course of the computation also has the property that its transitive closure has cardinality $\leq\aleph_{\omega\alpha+1}$. Since such sets can be encoded as subsets of $\aleph_{\omega\alpha+1}$, and the forcing does not add any subsets of $\aleph_{\omega\alpha+1}$, we have $\mathfrak{P}^{N[G]}(S)=\mathfrak{P}^{N}(S)$. Thus, $F$ will return (a code for) the same set that we would have obtained had the computation instead been performed with a Pot-function in $N$. Now, by assumption, in $N$, the result of the computation was $\aleph_{\omega\alpha+4}^{N}$. However, in $N[G]$, this is not even a cardinal, and certainly not the cardinality of the power set of $\aleph_{\omega\alpha+1}$. Thus, $(P_{k},\rho)$ does not witness the oW-reduction of PowerCard to Pot in $N[G]$. 

We note that the step just described ensures that $(P_{k},\rho)$ gets the cardinality of the power set of $\aleph_{\omega\alpha+1}$ wrong. Since no new subsets of $\aleph_{\omega\alpha+1}$ are added by this step, it will preserve the fact that $(P_{l},\xi)$ gets the result for $\aleph_{\omega(\omega\xi+l)+1}$ wrong for all $(P_{l},\xi)$ that preceed $(P_{k},\rho)$ in the ordering defined above. 

This explains one step of the iteration. We use iterated forcing with Easton support (\cite{Jech}, p. 395) to iterate it through all ordinals; since the iteration is progressively closed, it follows from \cite{Reitz}, Lemma 117 and Theorem 98, that this iteration yields a model of ZFC.
\end{proof}

What we have just seen thus means that PowerCard$\leq_{\text{OTM}}^{2}$Pot, where it is consistent with ZFC that this bound is optimal (but it is also consistent with ZFC that it is not).

\subsection{Reductions to DecCard}

In this section, we consider how many ordinals need to be checked for being cardinals for finding cardinal succcessors and for computing cardinalities of ordinals. 

\begin{theorem}{\label{nextcard to deccard}}

    Let $f:\text{On}\rightarrow\text{On}$ be the (class) function $f(\alpha)=\text{card}(\alpha)^{+}+1$. Then $f$ is the reduction complexity of NextCard to DecCard.
\end{theorem}
\begin{proof}

\begin{enumerate}
\item A reduction from NextCard to DecCard works as follows (see \cite{Ca2025Full}, Proposition 14): Given $\alpha\in\text{On}$, apply DecCard successively to all ordinals, starting with $\alpha+1$, until the answer is positive for some ordinal $\beta$; then return $\beta$. This works, and it clearly works within the required time bound.

\item Clearly, DecCard satisfies the assumption of Corollary \ref{cardinal time bound}. 
To see that $f$ is optimal, let $(P,\rho)$ witness the reduction, let $F$ be DecCard, and let $\alpha>\rho$ be infinite, but not a cardinal. Let $\kappa:=\text{card}(\alpha)^{+}$. Now assume for a contradiction that $P^{F}(\alpha,\rho)$ makes less than $\kappa+1$ many calls to $F$. This means that the number of calls to $F$ it makes is at most $\kappa$, and we already know from Corollary \ref{cardinal time bound} that that many calls are made prior to time $\kappa$. Thus, all calls to $F$ are made before time $\kappa$. But this means that all calls to $F$ evaluate $F$ at ordinals strictly less than $\kappa$; in particular, if $F$ is applied to an ordinal greater than $\alpha$, it always returns $0$. 

Let us modify $(P,\rho)$ a bit to work as follows: On input $\alpha$, it starts by successively calling $F$ for all $\xi\leq\alpha$ and storing the results on some extra tape. After that, $F$ is never used again; instead, we use the stored information to evaluate $F$ for ordinals $<\alpha$, while, if $F(\xi)$ is requested for some $\xi>\alpha$, we always return $0$. Using this, we can now simulate the computation of $P^{F}(\alpha,\rho)$ without actually using $F$ ever again. 

Now, this modified computation makes $\alpha+1<\kappa$ many calls to $F$ and thus, by Corollary \ref{cardinal time bound}, must halt in less than $\kappa$ many steps or will not halt at all. But, by assumption, it outputs $\kappa$, which means that it runs for at least $\kappa$ many steps before halting (for the read-write-head cannot reach cell $\kappa$ in fewer than $\kappa$ many steps), a contradiction. 
\end{enumerate}
\end{proof}

The naive approach to reducing OrdCard to DecCard explained in the proof of Proposition 14(4) in \cite{Ca2025Full} 
 takes $\alpha+1$ many steps in input $\alpha$ (in the worst case that $\alpha$ itself is a cardinal). A slight improvement would be to first check whether $\alpha$ is finite; if it is, return $\alpha$; and if it is not, start by applying DecCard to $\alpha$ and then to the ordinals strictly below $\alpha$, which would give us the new upper bound $\alpha$. One might conjecture that this is optimal. Surprisingly, it is consistent with ZFC that it is not at all:

\begin{proposition}{\label{in l ordcard to finite deccard}}
    If $V=L$, then OrdCard$\leq_{\text{OTM}}^{<\omega}$DecCard.
\end{proposition}
\begin{proof}
    Given an ordinal $\alpha$, the reduction works as follows: Use Koepke's algorithm to enumerate $L$ on an OTM (see, e.g., \cite{CarlBuch}, Lemma 3.5.3). Whenever a new $L$-level $L_{\beta}$ with $\beta>\alpha$ is produced, compute $\text{card}^{L_{\beta}}(\alpha)$ and store it on some extra tape. If that value is not already present on that tape, check it with DecCard. If the answer is positive, return that value; otherwise, continue with the next $\beta$. 

    This clearly yields the right result: If some $L$-level contains a bijection between some ordinal $\gamma$ and $\alpha$, and $\gamma$ is in fact an $L$-cardinal, then $\gamma$ is the $L$-cardinality of $\alpha$.

    Moreover, the sequence of values checked with DecCard is a strictly decreasing sequence of ordinals, and hence finite.
\end{proof}

We have thus seen that, in $L$, the reduction of OrdCard to DecCard takes only finitely many calls to DecCard. Provably in ZFC, this is the best one can hope for in any model; there cannot be a fixed finite bound of the required instances of DecCard:

 \begin{lemma}{\label{no finite guessing}}
 There is no parameter-program $(P,\rho)$ with the following property: There exists an $n\in\omega$ such that and a proper class $Q\subseteq\text{On}$ such that: 

 \begin{enumerate}
     \item For all $\alpha\in Q$, $P(\alpha,\rho)$ computes a finite set $q(\alpha)$ of ordinals with $|q(\alpha)|=n$ and $\text{card}(\alpha)\in q(\alpha)$. 
     \item For all ordinals $\alpha$, if $P(\alpha,\rho)$ computes a set of ordinals of cardinality $n$, then $\alpha\in Q$. 
 \end{enumerate}
 \end{lemma}
\begin{proof}
    Let $(P,\rho)$ be a parameter-program. We will show by induction on $n$ that $(P,\rho)$ does not have the above property for any natural number $n\geq 1$. 

    \begin{itemize}

    \item Let $n=1$. Assume for a contradiction that $(P,\rho)$ has the property in question. Consider the following program $P^{\prime}$: For a given ordinal $\alpha$, we start counting upwards, letting $P(\beta,\rho)$ run simultaneously for every $\beta$ we encounter (we imagine the tape split into $\text{On}$ many portions). By the assumption that $Q$ is a proper class (and thus unbounded in $\text{On}$), there are $\gamma>\beta>\alpha$ such that $\beta,\gamma\in Q$ and $\text{card}(\gamma)>\text{card}(\beta)\geq \text{card}(\alpha)^{+}$, which implies that $q(\beta)\cap q(\gamma)=\emptyset$. Thus, in the run of the computation, we are guaranteed to eventually find two ordinals $\gamma>\beta\geq\alpha$ for which $P(\beta,\rho)$ and $P(\gamma,\rho)$ terminate and output set $q(\beta)$, $q(\gamma)$ each containing a single ordinal, where $q(\beta)\cap q(\gamma)=\emptyset$. As soon as that happens, we halt. 

    Now, by definition of $(P,\rho)$, we have $\text{card}(\gamma)\in q(\gamma)$ and $\text{card}(\beta)\in q(\beta)$. Thus $q(\beta)\cap q(\gamma)\emptyset$ implies $\text{card}(\gamma)\neq \text{card}(\beta)$, and since $\gamma>\beta\geq \alpha$, it follows that $\text{card}(\gamma)\geq\text{card}(\alpha)^{+}$. 

    Now pick $\alpha>\rho$. Then the computation just described starts with the inputs $\alpha$ and $\rho$ and terminates with an output at least $\alpha^{+}$. In particular, it takes at least $\alpha^{+}$ many steps. But it is well-known that any halting OTM-computation in an infinite parameter $\xi$ must hold in less than $\xi^{+}$ many steps, a contradiction. 

   \item Now suppose the claim is true for $n$; we will show that it holds for $(n+1)$. Suppose for a contradiction that $(P,\rho)$ satisfies the above condition, computing sets of size $(n+1)$. 
   
   We first claim that, for any ordinal $\alpha>\rho$, the (finite) set $D_{\alpha}:=\bigcap_{\iota\geq\alpha,\iota\in Q}q(\iota)$ must be non-empty. 

   To see this, suppose otherwise, and pick $\alpha>\rho,\omega$ for which $D_{\alpha}=\emptyset$. This in particular implies that there is some $\delta\text{On}$ such that $\bigcap_{\alpha<\iota<\delta,\iota\in Q}q(\iota)=\emptyset$ (map any $\zeta\in q(\alpha)$ to the first ordinal $\delta(\zeta)\in Q$ such that $\zeta\notin q(\delta(\zeta))$, then take the maximum of these $\delta(\zeta)$ for all $\zeta\in q(\alpha)$). Starting with $\alpha$, we count upwards through the ordinals, letting $P(\beta,\rho)$ run simultaneously for any $\beta$ encountered on the way, until $q(\iota)$ for sufficiently many ordinals $\iota$ so that the intersection of all these $q(\iota)$ is empty. As soon as this happens, halt.  

   Now, since $\text{card}(\alpha)\in q(\iota)$ for any $\iota$ with $\text{card}(\iota)=\text{card}(\alpha)$, this means that we have counted up to some $\iota$ such that $\text{card}(\iota)>\text{card}(\alpha)$, thus running for at least $\alpha^{+}$ many steps in parameters $\alpha$ and $\rho$ (where $\rho<\alpha$), again a contradiction. 

   It follows that $D_{\rho+1}$ has at least one element, say, $\xi$. Now modify $(P,\rho)$ slightly to also use the parameter $\xi$ and, whenever $P(\beta,\rho)$ outputs a set of $(n+1)$ ordinals, delete $\xi$ from this set and return the result, while, if $P(\beta,\rho)$ outputs a set of $n$ or less ordinals, it returns the empty set. Then this new parameter-program computes, for each $\iota\in Q$ with $\iota>\text{max}\{\rho+1,\xi^{+}\}$, a set of $n$ ordinals guaranteed to contain $\text{card}(\iota)$ (for throwing out $\xi$ cannot remove the cardinality of $\iota$ from the set if that cardinality is greater than $\xi$). Moreover, it only outputs a set of $n$ ordinals when $(P,\rho)$ outputs a set of $(n+1)$ ordinals. But then, this modified program is exactly what is have ruled out by the inductive assumption about $n$. 
\end{itemize} 
\end{proof}

\begin{theorem}
    For any $n\in\omega$, OrdCard$\nleq^{n}_{\text{OTM}}$DecCard.
\end{theorem}

\begin{proof}
    Let $F$ be a DecCard-function, let $n\in\omega$, and suppose for a contradiction that $(P,\rho)$ is a parameter-program that witnesses the respective reduction. Let $F$ be a DecCard function. This means that, for any ordinal $\alpha$, there is some $s\in\{0,1\}^{n}$ such that $P^{s}(\alpha,\rho)\downarrow = \text{card}(\alpha)$. Let $\mathcal{C}\subseteq\mathfrak{P}(\{0,1\}^{n})$ be the set of subsets $T\subseteq\{0,1\}^{n}$ such that, for a proper class of ordinals $\alpha$, $P^{F\rightarrow t}(\alpha,\rho)$ terminates and outputs a single ordinal for all $t\in T$. Among the elements of $\mathcal{C}$, there is (at least) one of maximal cardinality $k$ (so $k\leq 2^{n}$); let $S$ be one of these. 
    By assumption on $(P,\rho)$, we have $S\neq\emptyset$ (for there must be at least one element in $\{0,1\}^{n}$ that represents the correct answers to the requests to $F$ cofinally often in $\text{On}$). 

    Let $\alpha>\rho$ be sufficiently large so that, for all $\beta\geq\alpha$, the set $\{t\in\{0,1\}^{n}:P^{F\rightarrow t}(\beta,\rho)\text{ halts and outputs a single ordinal}\}$ belongs to $\mathcal{C}$; in particular, these sets will all have cardinality at most $k$. Modify $(P,\rho)$ slightly to a parameter-program $(P^{\prime},\rho)$ that works as follows: On every input $\beta$, it simultaneously runs  $P^{F\rightarrow s}(\beta,\rho)$ for all $s\in S$. If all of these computations halt and output a single ordinal, we output the set $q(\beta)$ of these (at most) $k$ ordinals. Let $l\leq k$ be the maximal cardinal that occurs cofinally in $\text{On}$ often as the cardinality of such a set $q(\beta)$. Let $Q$ be the class of ordinals $\beta>\alpha$ such that $P^{F\rightarrow s}(\beta,\rho)$ halts and outputs a single ordinal for all $s\in S$ and such that $|q(\beta)|=l$. By assumption, $Q$ is a proper class. 

    Moreover, for $\beta\in Q$, we have $\text{card}(\beta)\in q(\beta)$. For if not, then this means that the sequence $s$ of correct answers to the requests that $(P^{\prime})^{F}(\beta,\rho)$ makes to $F$ is not contained in $S$, although $P^{F\rightarrow s}(\beta,\rho)$ terminates and outputs a single ordinal (for that output would, by definition of $(P,\rho)$, be $\text{card}(\beta)$). But that means that $P^{t}(\beta,\rho)$ actually halts and outputs a single ordinal for all $t\in S\cup\{s\}$ for an ordinal $\beta>\alpha$, contradicting the maximality of $S$.

    It follows that $(P^{\prime},\rho)$ and $Q$ satisfy the conditions ruled out by Lemma \ref{no finite guessing}; thus, $(P^{\prime},\rho)$, and hence $(P,\rho)$, cannot exist. 
\end{proof}

The algorithm described in the proof of Proposition \ref{in l ordcard to finite deccard} will work in general when $V=L[a]$ in the oracle $a$ when $a$ is a set of ordinals. 
If $V$ is very much unlike $L$, and there is no extra oracle, however, this will not be true:

\begin{proposition}
If $0^{\sharp}$ exists, then  OrdCard$\nleq_{\text{OTM}}^{<\omega}$DecCard.
\end{proposition}
\begin{proof}
If $0^{\sharp}$ exists, then the $V$-cardinals are order indiscernibles for $L$.\footnote{See, e.g., \cite{Jech}, Theorem 18.1(ii).} Assume for a contradiction that $(P,\rho)$ is parameter-program that witnesses OrdCard$\leq_{\text{OTM}}^{<\omega}$DecCard. Pick a Silver indiscernible $\xi>\rho,\aleph_1$ which is not a $V$-cardinal, and let $F$ be DecCard. By assumption, $P^{F}(\xi,\rho)$ computes $\text{card}(\xi)$ and uses $F$ only finitely often along the way.  This will in particular reveal only finitely many cardinals; let us say that $s:=(\kappa_{i}:i<n)$ is the sequence of cardinals found along the way, where $n\in\omega$. Then we can view the computation as running relative to a function that returns $1$ on elements of $s$ and $0$ everywhere else; thus, the fact that $P^{F}(\xi,\rho)\downarrow=\text{card}(\xi)$ can be expressed as a first-order formula $\phi(\kappa_{0},...,\kappa_{n-1},\text{card}(\xi))$. Due to absoluteness of computations, this formula will be absolute between $L$ and $V$. However, since $\text{card}(\xi)$ is an uncountable cardinal, the class of Silver indiscernibles is unbounded in $\text{card}(\xi)$; thus, there will be a Silver indiscernible $\beta$ such that $\text{card}(\xi)>\beta>\text{max}(\{\kappa\in s:\kappa<\text{card}(\xi)\})$. It follows that $L\models\phi(\kappa_{0},...,\kappa_{n-1},\beta)$; but the computation $P^{F}(\xi,\rho)$ cannot halt with two different outputs, a contradiction.
\end{proof}

\begin{remark}
We note that, in $L[0^{\sharp}]$, regardless of the input, we can get away with $<\aleph_{\omega}+\omega$ calls to DecCard. This works by first running through the first $\aleph_{\omega}+1$ ordinals, checking all of them with DecCard until $\omega+1$ many cardinals have been found (which will be the cardinals $\aleph_{0},\aleph_{1},...,\aleph_{\omega}$). Then $(\aleph_{i}:i<\omega)$ is an infinite set of Silver indiscernibles, and we have $L_{\aleph_{\omega}}\prec L$. Thus, by computing a code for $L_{\aleph_{\omega}}$ from $\aleph_{\omega}$ and evaluating formulas with parameter $\aleph_{1},\aleph_{2},...$ in $L_{\aleph_{\omega}}$, we can compute $0^{\sharp}$. But then, as described above, relative to $0^{\sharp}$, we only need finitely many extra calls to DecCard to compute the cardinality of any given ordinal. Thus, in $L[0^{\sharp}]$, we still have a constant upper bound to the number of necessary calls.

Note that this construction also works relative to iterated sharps: For example, in order to evaluate OrdCard in $L[(0^{\sharp})^{\sharp}]$, we first determine $\aleph_{1},...,\aleph_{\omega}$, as above; then, on the basis of this, we compute $0^{\sharp}$ by evaluating the truth predicate in $L_{\aleph_{\omega}}$; and then, we compute $(0^{\sharp})^{\sharp}$ by evaluating the truth predicate in $L_{\aleph_{\omega}}[0^{\sharp}]$. 

We do not know whether this bound is optimal, but conjecture that it is not.

\end{remark}

At least consistently, there need not be a constant bound on the reduction complexity of OrdCard to DecCard:

\begin{theorem}{\label{no constant complexity bound}}
There is a class forcing extension of $L$ which satisfies ZFC such that OrdCard$\nleq_{\text{OTM}}^{<\alpha}$DecCard for every ordinal $\alpha$.
\end{theorem}
\begin{proof}

For each triple $(P,\rho,\alpha)$ consisting of a parameter-program $(P,\rho)$ and an ordinal $\alpha$, we sabotage the claim that $(P,\rho)$ reduces OrdCard to DecCard with complexity bounded above by $\alpha$. Let $T$ be the class of all such triples $(P,\rho,\alpha)$, and let $\leq^{T}$ be the ordering on $T$ induced by Cantor's pairing function (used twice to encode triples); this is a linear ordering of order type $\text{On}$.

With each triple $t=(P,\rho,\alpha)\in T$, we associate an ordinal $\kappa(t)$ so that $\kappa(t)$ is an uncountable limit cardinal in $L$ and such that $$\text{cf}(\kappa(t))>\lambda(t):=(\text{sup}(\{\kappa(t^{\prime})^{+}:t^{\prime}<^{T}t\})\cup\rho+1\cup\alpha+1)^{+}.$$

Note that, by Corollary \ref{cardinal time bound}, if $F$ is DecCard (which is clearly definable and does not raise cardinals, as it only outputs $0$ or $1$), if $P^{F}(\kappa(t),\rho)$ stopped after at least $\kappa(t)^{+}$ many steps, it would have made at least $\kappa(t)^{+}>\alpha$ many calls to $F$ at time $\kappa(t)^{+}$, and thus have violated the supposed upper bound $\alpha$ to the number of these calls. Thus, we only need to take care of cases in which $P^{F}(\kappa(t),\rho)$ halts in less than $\kappa(t)^{+}$ many steps -- in all other cases, it is either guaranteed to make more calls to the extra function, or it is guaranteed not to halt.

The desired target model will arise as an iterated forcing of class length $\text{On}$ using Easton support. Suppose that, for some  $t\in T$, an intermediate model $M_{<t}$ has been obtained that takes care of all $t^{\prime}<t=(P,\rho,\alpha)\in T$. The forcing will be set up in a way that taking care of $t$ does not change cardinals or cofinalities $\leq\kappa(t^{\prime})^{+}$ for all $t^{\prime}<t$. ($\ast$) 
It will, moreover, not collapse cardinals greater than $\kappa(t)$. ($\ast\ast$) 
Thus, in particular, if $\kappa$ satisfies the definition of $\kappa(t)$ in the ground model, it will continue to do so in the generic extension: All forcings for $t^{\prime}<^{T}t$ will leave $\lambda(t)$ intact. 

Let $F$ be DecCard in $M$. Our aim is to ensure that, in the target model, the cardinality of $\kappa(t)$ is not computed correctly with less than $\alpha$ many calls to $F$. We define $G:\text{On}\rightarrow\text{On}$ by $$G(\xi):=\begin{cases}0\text{, if }\xi=\kappa(t)\\F(\xi)\text{, otherwise}\end{cases}.$$

We now consider the computation of $P^{G}(\kappa(t),\rho)$. We distinguish the following cases:

\begin{enumerate}
\item Before time $\kappa(t)^{+}$, the computation $P^{G}(\kappa(t),\rho)$ contains fewer than $\alpha$ many calls to $G$ and has not halted.
\item Before time $\kappa(t)^{+}$, the computation $P^{G}(\kappa(t),\rho)$ contains fewer than $\alpha$ many calls to $G$ and it has halted.
\item The computation $P^{G}(\kappa(t),\rho)$ has made at least $\alpha$ many calls to $G$ before time $\kappa(t)^{+}$. 
\end{enumerate}

Note that, since $\kappa(t)>\rho$ and because $F$ does not raise cardinals, the program can write on at most one extra cell on each tape per time step, so that, whatever it can write in less than $\kappa(t)^{+}$ many steps in inputs $\rho$ and $\kappa(t)$ will have cardinality less than $\kappa(t)^{+}$; in particular, all calls to $G$ the computation will make before time $\kappa(t)^{+}$ will concern ordinals less than $\kappa(t)^{+}$. 

We now let \begin{multline}\delta:=\text{sup}(\{\iota<\kappa(t):\text{Among the first }\alpha\text{ calls to }G\text{ before time }\kappa(t)^{+}\\\text{ in the computation }P^{G}(\kappa(t),\rho)\text{ one concerns }\iota\}).\end{multline}

Since $\text{cf}(\kappa(t))>\alpha$ by definition, we have $\delta<\kappa(t)$. We now let $\lambda$ be the smallest $M_{<t}$-cardinal in $((\text{max}(\delta,\lambda(t)^{+}),\kappa(t))$ which is not equal to the output of $P^{G}(\kappa(t),\rho)$. If there is no such output, this is trivial; if there is, $\kappa(t)$ being a limit cardinal guarantees the existence of $\lambda$. 

We now force with the Levy collapse forcing $\text{Coll}(\kappa(t),\lambda)$, which consists of the partial functions from $\kappa(t)$ to $\lambda$ of cardinality $<\lambda$, ordered by $\supseteq$. 

Since this is $\lambda$-closed, no cardinals below $\lambda$ will be changed; thus, ($\ast$) is satisfied. Moreover, it satisfies the $\kappa^{+}$-cc, 
and thus does not collapse cardinals $\geq\kappa(t)^{+}$, so that ($\ast\ast$) is satisfied as well. 

Let $M_{t}$ be the generic extension, and let $F^{\prime}$ be DecCard in $M_{t}$. We now show for each of the three cases that, in $M_{t}$, it is not true that $P^{F^{\prime}}(\kappa(t),\rho)\downarrow = \text{card}(\kappa(t))$ with less than $\alpha$ many calls to $F^{\prime}$. 

Let $\gamma$ be the time before $\kappa(t)^{+}$ at which the $\alpha$-th call to $G$ takes place in the computation of $P^{G}(\kappa(t),\rho)$ if such a time exists, and $\gamma:=\kappa(t)^{+}$, otherwise. 
The crucial observation is that, as seen above, (i) this forcing does not change cardinals $\leq\delta$, (ii) no calls concerning ordinals in $(\delta,\kappa(t))$ are made to $G$ before time 
$\gamma$ in the computation of $P^{G}(\kappa(t),\rho)$, (iii) we have $F^{\prime}(\kappa(t))=0$ since $\kappa(t)$ is collapsed and (iv) both $F^{\prime}$ and $G$ return $0$ for every ordinal in $(\kappa(t),\kappa(t)^{+})$. Thus, $F^{\prime}$ will return the same value as $G$ for any ordinal to which it is applied in the computation before time $\gamma$. 
Consequently, the computation of $P^{G}(\kappa(t),\rho)$ and $P^{F^{\prime}}(\kappa(t),\rho)$ in fact agree up to time $\gamma$. 

Now, in case (1), we know that $P^{G}(\kappa(t),\rho)$ will not halt and it will make no calls to $G$ at or after time $\kappa(t)^{+}$. Thus, the computations $P^{G}(\kappa(t),\rho)$ and $P^{F^{\prime}}(\kappa(t),\rho)$ in fact agree entirely in this case, so $P^{F^{\prime}}(\kappa(t),\rho)\uparrow$. 

In case (2), the output of $P^{F^{\prime}}(\kappa(t),\rho)$ is the same as that of $P^{G}(\kappa(t),\rho)$ (since the computations agree). But in $M_{t}$, we have $\text{card}(\kappa(t))=\lambda$, which, by choice of $\lambda$, is not the output of $P^{G}(\kappa(t),\rho)$. 

In case (3), $P^{F^{\prime}}$ has made at least $\alpha$ many calls to $F^{\prime}$ before time $\gamma$ 
and thus does not adhere to the supposed bound on the number of calls. 

This sequence of forcings is progressively closed. Thus, again by Reitz \cite{Reitz} Lemma 117 and Theorem 98, the iteration yields a model of ZFC. 
\end{proof}

\begin{remark}
    The above proof only invokes rather general properties of DecCard; it thus applies at least to every class function instead of DecCard mapping to $\{0,1\}$, and in fact to a considerably wider range of class functions.
\end{remark}

The following lemma roughly says that many small ordinals must be considered in a computation relative to DecCard before a large one is considered.

\begin{lemma}{\label{many requests precede high requests}}
Let $(P,\rho)$ be a parameter-program, let $H$ be a DecCard function, $\beta\in\text{On}$, and let $\kappa>\rho,\beta$ be a cardinal. If $P^{H}(\beta,\rho)$ makes a request to $H$ concerning some ordinal $\alpha\geq\kappa^{+}$, then it has made at least $\kappa^{+}$ requests to $H$ before.
\end{lemma}
\begin{proof}
    Suppose otherwise. Let $\xi$ be the order type of the set of times at which calls to $H$ were made before $H(\alpha)$ was requested; so we assume $\xi<\kappa^{+}$. By reflection, let $\nu>\kappa^{+}$ be such that $V_{\nu}$ is a $\Sigma_{2}$-elementary submodel of $V$. In particular, an ordinal $\eta<\nu$ is a cardinal if and only if it is a cardinal in $V_{\nu}$. Let $\mathcal{H}$ be the elementary hull of 
    $(\kappa+1)\cup (\xi+1)$ in $V_{\nu}$, and let $\overline{\mathcal{H}}$ be its transitive collapse. We have $\text{card}(\mathcal{H})=\kappa$, so the ordinal height of $\overline{\mathcal{H}}$ is strictly less than $\kappa^{+}$. Moreover, the collapsing map will leave $\xi$ and $\kappa$ unchanged. Moreover, by elementarity, an ordinal is a cardinal in $\overline{\mathcal{H}}$ if and only if it is a cardinal in $V$. Thus the computation of $P^{H}(\beta,\rho)$ will be absolute between $\overline{\mathcal{H}}$ and $V$ up to the time where the $\xi$-th call to $H$ takes place. Again by elementarity, $\overline{\mathcal{H}}$ believes that there this computation makes a $\xi$-th call to $H$. But that call would have concern $\alpha$, which cannot be an element of $\overline{\mathcal{H}}$, a contradiction. 
\end{proof}

\begin{remark}
    Considering the \textit{first} call to some ordinal $\alpha\geq\kappa^{+}$ in the last lemma, we thus obtain that at least $\kappa^{+}$ many requests to $H$ have been made that concern ordinals strictly smaller than $\kappa^{+}$ prior to the first call to a larger ordinal. 
\end{remark}

%hierher

We have already seen that the identity function is an upper bound on the reduction complexity of OrdCard to DecCard, as you can always start by testing the given ordinal $\alpha$ and then all ordinals below it successively. It is then natural to ask whether this is optimal: 
Is there (provably in ZFC) a reduction of OrdCard to DecCard which provably in ZFC improves on the naive approach explained above in the sense that, for some $\alpha\in\text{On}$, the number of calls required on input $\xi>\alpha$ will be strictly smaller than $\xi$? As the next theorem shows, this is not the case; it is thus consistent relative to ZFC that the upper bound $\text{id}$ is in fact optimal.\footnote{As a result, this is of course a strengthening of Theorem \ref{no constant complexity bound}. However, the proof method here is much more specific to reductions to DecCard, which is why Theorem \ref{no constant complexity bound} -- or rather its proof -- is of independent interest.}

\begin{theorem}{\label{ordcard to deccard not faster than id}}
    It is consistent relative to ZFC that any reduction of OrdCard to DecCard will, for unboundedly many $\alpha$, require at least $\alpha$ many calls to DecCard, i.e., that  $\text{OrdCard}\nleq_{\text{OTM}}^{<\text{id}}\text{DecCard}$. 
\end{theorem}
\begin{proof}
As explained above, we identify parameter-programs with ordinals. For $\iota\in\text{On}$, we write $\kappa_{\iota}$ for the $\iota$-th fixed point of the $\aleph$ function in $L$. We will avoid complications for our construction that would arise from such fixed points that are themselves limits of such fixed points and thus only consider successors in this sequence; accordingly, we write $\kappa(\alpha,\beta)$ for $\kappa_{p(\alpha,\beta)+1}$. Thus, we have reserved, for each parameter-program $\alpha$, a proper class $C_{\alpha}$ of fixed points of the $\aleph$ function in $L$ (which are not limits of such fixed points), namely $C_{\alpha}:=\{\kappa(\alpha,\iota):\iota\in\text{On}\}$. We assume without loss of generality (e.g., by replacing $\kappa(\alpha,\beta)$ with $\kappa(\alpha,\beta+\alpha)$ if necessary) that $\kappa(\alpha,\beta)>\text{max}(\alpha,\beta)$. 

We now proceed by an iterated forcing, sabotaging, for each pair $(\gamma,\alpha)$ of a (coded) parameter-program $\gamma$ and an ordinal $\alpha$ the claim that, for all $\delta>\alpha$, $\gamma$ achieves a reduction of OrdCard to DecCard with strictly less than $\delta$ many calls to DecCard. Noting that DecCard is easily definable as a class function, let us denote by $H^{M}:\text{On}\rightarrow\{0,1\}$ an (in fact, the unique) DecCard function in the transitive ZFC-model $M$; when the model is clear from the context, the superscript $M$ will be dropped. By replacing the parameter-program $\gamma$ with the parameter-program $\gamma^{\prime}$ guaranteed to exist by Lemma \ref{no repeated requests}, we can assume without loos of generality that each parameter-program $\gamma$ we consider calls $H$ at most once on each ordinal. 

We now explain one step of the iteration, which will start with $L$. Let an ordered pair $(\gamma,\alpha)$ of a parameter-program $\gamma$ and an ordinal $\alpha$ be given. Consider the cardinal $\kappa:=\kappa(\gamma,\alpha)$. Since $\kappa$ is by definition not a limit of fixed points of the $\aleph$ function in $L$, there is a largest such fixed-point $\kappa^{\prime}$ below $\kappa$. Let $\delta\in((\kappa^{\prime})^{+},\kappa)$ be an $L$-cardinal, and let define the class function $\overline{H}$ as follows, for all $\xi\in\text{On}$: 
$$\overline{H}(\xi)=\begin{cases}0\text{, if }\delta<\xi\leq\kappa\\H(\xi)\text{, otherwise.}\end{cases}$$

We assume that our construction has already taken care of all earlier pairs in the natural well-ordering of pairs of ordinals and that we have thus obtained a transitive $M\models\text{ZFC}$ with the following properties: Denoting by $\nu$ the supremum of the fixed points of the $\aleph$-function below $\kappa$ in $L$ -- which, by assumption, will be strictly smaller than $\nu$ -- then, from $(\nu{+})^{M}$ upwards, the $M$-cardinals are the same as the $L$-cardinals. We will now pass to a generic extension $M[G]$ of $M$ with the following properties:

\begin{enumerate}
\item cardinals up to and including $(\nu^{+})^{M}$ are not changed 
\item cardinals greater than or equal to $(\kappa^{+})^{M}$ are not changed (and thus remain the same as in $L$)
\item $\kappa$ is collapsed 
\item for every cardinal $\mu\geq\kappa^{+}$, we have $2^{\mu}=\mu^{+}$ in $M[G]$ 
\item the computations $\gamma^{\overline{H}}(\kappa)$ and $\gamma^{H^{M[G]}}(\kappa)$ agree
\item either the computation $\gamma^{\overline{H}}(\kappa)$ makes %calls to all elements of a final segment of $\text{Card}^{L}\cap\kappa$ (and thus 
at least $\kappa$ many calls to $\overline{H}$, or it does not halt with output $\text{card}^{M[G]}(\kappa)$. 
\end{enumerate}

Let $S$ be the set of ordinals to which the computation $\gamma^{\overline{H}}(\kappa)$ makes calls. 

We distinguish the following cases: 

\begin{case_dist}
For some $\zeta<\kappa$, $(\zeta,\kappa)\cap S=(\zeta,\kappa)\cap\text{Card}^{L}$.
\end{case_dist}

By assumption that $\kappa$ is a fixed-point of the $\aleph$ function, this implies in particular that $\gamma^{\overline{H}}(\kappa)$ makes at least $\kappa$ many calls to $\overline{H}$. In this case, $M[G]$ is obtained by collapsing $\kappa$ to $\delta$. This implies (1), (2), (3) and (4). 
Consequently, we have $\overline{H}=H^{M[G]}$ and thus (5). If this computation does not call $\overline{H}$ for any ordinal greater than or equal to $(\kappa^{+})^{M}$, then, because of (1), further extensions will not change it and we are done. If it does call $\overline{H}$ for some such ordinal, then, by Lemma \ref{many requests precede high requests}, at least $(\kappa^{+})^{M}>\kappa$ many requests to $\overline{H}$ concerning ordinals strictly below $(\kappa^{+})^{M}$ have been made prior to it, and further extensions will not change the answers to these requests, so in this case as well, the computation $\gamma^{H}(\kappa)$ is guaranteed to make more than $\kappa$ many calls to $H$, which suffices. In both cases, we thus have the first option of (6).

\begin{case_dist}
Otherwise.
\end{case_dist}

Thus, there are cofinally in $\kappa$ many $L$-cardinals for which $\overline{H}$ is never called in the computation of $\gamma^{\overline{H}}(\kappa)$. Pick one that is greater than $\delta$ and different from the output of this computation (if there is one); call it $\lambda$. 

Now, $M[G]$ is obtained by (i) collapsing $\kappa$ to $\lambda$ and (ii) collapsing all $L$-cardinals in $(\delta,\lambda)$ to $\delta$. Again, (1), (2), (3) and (4) are standard. 
To see (4), just note that the only argument for which $\overline{H}$ and $H^{M[G]}$ differ is $\lambda$, which, however, by our assumption, is not considered by this computation. Moreover, the construction guarantees that $\text{card}^{M[G]}(\kappa)=\lambda$ and that $\lambda$ is not the output of $\gamma^{H^{M[G]}}(\kappa)$, so we have the second alternative of (6). 

In either case, we have (6): $\kappa$ is a counterexample to the claim that $\rho$ reduces OrdCard to DecCard on all instances above $\alpha$ in reduction complexity strictly less than $\text{id}$. 

For the iteration, we again use Easton support. And again, since the iteration is progressively closed, it yields a model of ZFC in which $\text{OrdCard}\nleq_{\text{OTM}}{<\text{id}}\text{DecCard}$. 

\end{proof}

\begin{remark}
This does \textit{not} mean that it is consistent that a reduction of OrdCard to DecCard will have to make calls to DecCard for \textit{all} ordinals below the input; for example, it is clearly sufficient to only make requests for limit ordinals, or admissible ordinals, or admissible limits of admissible ordinals etc., and answer all requests to other ordinals with $0$ without invoking the DecCard function.
\end{remark}

We end this section with two observations on the cofinality operators. 

\begin{lemma}{\label{cofinality basics}}
    \begin{enumerate}
        \item $\text{DecReg}\leq_{\text{oW}}\text{Cof}$. 
        \item If $V=L$, then $\text{Cof}\leq^{<\omega}\text{DecReg}$.
    \end{enumerate}
\end{lemma}
\begin{proof}
    \begin{enumerate}
    \item Given $\alpha$, check whether $\text{Cof}(\alpha)=\alpha$. 
    \item Recall that $\text{cf}(\alpha)$ is (not just the smallest, but) the \textit{unique} regular cardinal $\kappa$ for which there is an unbounded map $f:\kappa\rightarrow\alpha$. Now, starting with $\beta=\alpha+1$, compute $\text{cf}^{L_{\beta}}(\alpha)$; for the initial value, and whenever that value changes, use DecReg to check whether it is a regular cardinal; once it is, we halt and use it as the output of our computation. Since the value can only fall when it changes, and can only fall finitely often, this uses DecReg only finitely often. 
    \end{enumerate}
\end{proof}

\subsection{On separation and truth}
\label{separation and truth}

In Lemma 11 of \cite{Ca2025Full}, 
we showed that $\Sigma_{n}$-separation is reducible to $\Sigma_{n}$-truth 
using at least\footnote{Note that, since the computation works through a given code for $a$, which may  order $a$ in a non-minimal way, it may well make more such calls in terms of the order-type.} $\text{card}(a)$ applications in input $a$, using the obvious idea: Run through the given set and test each element with the truth predicate for satisfying the property in question. But are that many calls to truth really necessary? This is the question we will treat in this section. 

We begin by noting that a finite number of calls to any $\Sigma_{k}$-truth predicate -- and thus, in particular, a single such call -- is not enough.

\begin{corollary}{\label{sep to truth with finitely many calls}}
    There is no $n\in\omega$ such that $\Sigma_{1}$-Sep$\leq_{\text{OTM}}^{<\omega}\Sigma_{n}$-truth.
\end{corollary}
\begin{proof}
    Assume for a contradiction that $(P,\rho)$ witnesses such a reduction of $\Sigma_{1}$-Sep to $\Sigma_{n}$-truth, for some $n\in\omega$. Let $h_{\rho}:=\{i\in\omega:P_{i}(\rho)\downarrow\}$ be the OTM-halting problem in the parameter $\rho$. Then $h_{\rho}$ is a subset of $\omega$. Let $\phi(\rho)$ be the $\Sigma_{1}$-formula that defines $h_{\rho}$ as a subset of $\omega$ in the parameter $\rho$, and let $F$ be a $\Sigma_{n}$-truth-function. Now, by assumption, only finitely many calls are made to $F$ in the computation of $P^{F}(\omega,\rho)$. 
    But this means that the sequence $\vec{s}$ of the finitely many outcomes can be hardcoded in a variant $Q$ of $P$ that, when $P$ calls $F$ for the $j$-th time, just uses the $j$-th bit of $\vec{s}$ as the result to continue. Thus, $Q$ is an OTM-program which, in the parameter $\rho$, computes $h_{\rho}$, i.e., solving the halting problem for OTM-programs in the parameter $\rho$, a contradiction. 
\end{proof}

 The following lemma summarizes the main idea behind the argument:

\begin{lemma}
    Let $F$, $G$ be functions mapping sets of ordinals to sets of ordinals. Assume that there is a parameter-program $(P,\rho)$ such that, for some set $a$, $P^{F}(a,\rho)$ computes $G(a)$ and makes only finitely many calls to $F$.
    \begin{enumerate}
        \item If $F(x)\in L$ for all $x$ (i.e., $F(x)$ is parameter-OTM-computable), then $G(a)\in L$.
        \item If $F(x)$ is OTM-computable in the parameter $\rho$ for all $x$, then $G(a)$ is OTM-computable in the parameter $\rho$ and the input $a$, for every input $a$.
\end{enumerate}
\end{lemma}
\begin{proof}
We only show (2); the proof for (1) is completely analogous. 

Let $\vec{v}=(v_{1},...,v_{k})$ be the sequence of values that $F$ returns in the finitely many calls to $F$. By assumption, let $Q_{1},...,Q_{k}$ be OTM-programs that compute $v_{1},...,v_{k}$ in the parameter $\rho$, respectively. Then we can modify $P$ to work as follows: For $i\leq k$, in the $i$-th call to $F$, it runs $Q_{i}(\rho)$ and uses the output as the return value of $F$. The computation will be (as a sequence of computational states) identical to that of $P^{F}(a,\rho)$, and thus have the same output; but it also a computation that only uses the input $a$ and the parameter $\rho$.    
\end{proof}

Intuitively, one should expect that there is no better way to obtain separated sets from truth than considering each element of the given set separately. 
Under a moderate extra assumption, we can indeed prove this:

\begin{lemma}
    Assume that there is a definable global well-ordering $\leq^{\ast}$ of $V$ which is compatible with the $\in$-relation.\footnote{This is equivalent to assuming $V=\text{HOD}$.} Then the following is true: Let $f:V\rightarrow\text{On}$ be a class function such that, for all but set many values of $x$, we have $f(x)<\text{card}(x)$. 
    Then, for no $m\in\omega$ we have $\Sigma_{4}$-Sep$\leq_{\text{OTM}}^{f}\Sigma_{m}$-truth.
\end{lemma}
\begin{proof}
    Let $f$ be as in the assumption. We can assume without loss of generality\footnote{Cf., e.g., Hamkins, \cite{Hamkins:MO}.} that $\leq^{\ast}$ is $\Sigma_{2}$-definable. 
    Assume for a contradiction that, for some $m\in\omega$, some parameter-program $(P,\rho)$ witnesses $\Sigma_{4}$-Sep$\leq_{\text{OTM}}^{f}\Sigma_{m}$-truth. Pick an uncountable cardinal $\kappa>\rho$ such that $2^{<\kappa}=\kappa$ and and let $\leq^{\prime}$ be the $\leq^{\ast}$-smallest well-ordering of $\mathfrak{P}^{<\kappa}(\kappa)$ in order type $\kappa$.\footnote{To see that there are unboundedly many such $\kappa$, note that, defining $\alpha_{0}:=\aleph_{1}$, $\alpha_{\iota+1}:=2^{\alpha_{\iota}}$ and $\alpha_{\lambda}:=\bigcup_{\iota<\lambda}\alpha_{\iota}$ for a limit ordinal $\lambda$, each fixed point of the normal function $\iota\mapsto\alpha_{\iota}$ will have this property.} 
    Let $g:\kappa\rightarrow\mathfrak{P}^{<\kappa}(\kappa)$ be the enumeration induced by $\leq^{\prime}$, and define $h:\kappa\rightarrow\omega\times\mathfrak{P}^{<\kappa}(\kappa)$ as $h(\omega\iota+k)=(k,f(\iota))$ for $\iota<\kappa$, $k\in\omega$. 

Define a subset $S\subseteq\kappa$ as follows: For $\iota<\kappa$, we have $\iota\in S:\Leftrightarrow \neg P_{h(\iota)_{0}}^{F\rightarrow h(\iota)_{1}}(\rho)\downarrow=1$.  $S$ is clearly definable as a subset of $\kappa$, and the definition is $\Sigma_{4}$.

Now, by assumption, $P^{F}(\rho)$ computes $S$, making $\xi<\kappa$ many calls to $F$. Let $\vec{v}:=(i_{\iota}:\iota<\xi)$ be the sequence of values returned by $F$ to these requests. Thus, $P^{F\rightarrow\vec{v}}(\rho)$ computes $S$ as well. Clearly, $\vec{v}$ can be regarded as (corresponding to) an element of $\mathfrak{P}^{<\kappa}(\kappa)$. We can modify $P$ to a program $Q$ -- in the same parameters -- which, rather than writing $S$ to the tape and halting, takes as an additional input some $\iota<\kappa$ and decides whether $\iota\in S$. Let $k$ be the index of $Q$ in the enumeration of programs, and let $\alpha<\kappa$ be the pre-image of $(k,\vec{v})$ under $h$. Then $P_{h(\alpha)_{0}}^{F\rightarrow h(\alpha)_{1}}(\rho)\downarrow=1 \Leftrightarrow \alpha\in S\Leftrightarrow\neg P_{h(\iota)_{0}}^{F\rightarrow h(\iota)_{1}}(\rho)\downarrow=1$, a contradiction. 

\end{proof}

\begin{question}
    Can the assumption of a definable global well-ordering be eliminated from the last result? We conjecture that this is the case.
\end{question}

\section{Conclusion and further work}

Clearly, there are many other principle that could be meaningfully investigated with respect to reduction complexity. 

While most results in this paper should be conceptually stable under changes of the underlying model of computation, some of them might allow for refinements that are more to the point. The reducibility concept defined and applied in this paper allows formalizing of intuitively meaningful and natural questions such as ``how many applications of power set are needed in order to calculate the cardinality of power sets?''. However, there are some questions of this kind for which the answer given by our formalization is not quite satisfying. A typical example would be ``how many applications of power set are needed in order to calculate cardinalities?''. The answer given in \cite{Ca2025Full}, Proposition $14$ -- that one application is enough -- depends heavily on the fact that, since sets need to be encoded before an OTM can operate on them, every set given to an OTM comes with a well-ordering. For such questions, it should be interesting to study similar reducibility notions on models of transfinite computability that can compute directly on sets, rather than on encodings of sets; Passmann's ``Set Register Machines'' introduced in \cite{Passmann} would be an example of such a notion.

\end{document}